\documentclass[12pt]{amsart}
\usepackage{amsmath}
\usepackage{amsfonts,mathtools}
\usepackage{amssymb,color}
\usepackage{hyperref}
\usepackage{amsthm,doi}
\usepackage{thmtools}
\usepackage{tikz,color}
\usepackage{pgfplots}
\usepackage{enumitem}
\pgfplotsset{compat=1.15}

\hypersetup{
	colorlinks   = true, 
	urlcolor     = blue, 
	linkcolor    = blue, 
	citecolor   = red 
}

\newcommand{\ip}[2]{\ensuremath{\left<#1,#2\right>}}
\newcommand{\sett}[1]{\ensuremath{\left \{ #1 \right \}}}
\newcommand{\abs}[1]{\ensuremath{\left| #1 \right| }}

\newcommand{\cu}{\kappa_{\partial\Omega,\eta}}

\newcommand{\N}{\mathbb{N}}

\newcommand{\R}{\mathbb{R}}

\newcommand{\hh}{\mathbb{H}}

\DeclareMathOperator{\im}{Im}
\DeclareMathOperator{\tr}{trace}

\DeclarePairedDelimiter\ceil{\lceil}{\rceil}
\newcommand{\rchi}{1}

\newtheorem{lemma}{Lemma}[section]
\newtheorem{theorem}[lemma]{Theorem}
\newtheorem{coro}[lemma]{Corollary}
\newtheorem{prop}[lemma]{Proposition}

\newtheorem{rem}[lemma]{Remark}

\numberwithin{equation}{section}

\newcommand{\lo}{L}
\newcommand{\newkg}{K_g}

\author{Felipe Marceca}
\address{Faculty of Mathematics \\
	University of Vienna \\
	Oskar-Morgenstern-Platz 1 \\
	A-1090 Vienna, Austria}
\email{felipe.marceca@univie.ac.at}

\author{Jos\'e Luis Romero}
\address{Faculty of Mathematics \\
	University of Vienna \\
	Oskar-Morgenstern-Platz 1 \\
	A-1090 Vienna, Austria \\and
	Acoustics Research Institute\\ Austrian Academy of Sciences\\Wohllebengasse 12-14, Vienna, 1040, Austria}
\email{jose.luis.romero@univie.ac.at}

\title[Spectral deviation of concentration operators]{Spectral deviation of concentration operators for the short-time Fourier transform}

\thanks{J. L. R. and F. M. gratefully acknowledge support from the Austrian Science Fund (FWF): Y 1199.}

\subjclass[2010]{47B35, 47G30, 42B35, 47A75, 42C40}

\keywords{Short-time Fourier transform, concentration operator, Hankel operator, eigenvalue}

\begin{document}
\begin{abstract}
Time-frequency concentration operators restrict the integral analysis-synthesis formula for the short-time Fourier transform to a given compact domain. We estimate how much the corresponding eigenvalue counting function deviates from the Lebesgue measure of the time-frequency domain. For window functions in the Gelfand-Shilov class, the bounds almost match known asymptotics, with the advantage of being effective for concrete domains and spectral thresholds. As such our estimates allow for applications where the spectral threshold depends on the geometry of the time-frequency concentration domain. We also consider window functions that decay only polynomially in time and frequency.
\end{abstract}

\maketitle

\section{Introduction}
\subsection{Spectrum of concentration operators}
The short-time Fourier transform of a function $f \in L^2(\mathbb{R}^d)$ is defined by means of a smooth and fast decaying \emph{window function} $g:\mathbb{R}^d \to \mathbb{C}$ as
\begin{align}
\label{eq_stft}
V_{g}f(z)
=\int_{\mathbb{R}^d}f(t)\overline{g(t-x)}e^{-2\pi i\xi t}dt,
\quad z=(x,\xi) \in \mathbb{R}^d \times \mathbb{R}^d.
\end{align}
The number $V_{g}f(x,\xi)$ measures the importance of the frequency $\xi$ of $f$ near $x$ and provides a form of local Fourier analysis.
Indeed, with the normalization $\lVert{g \rVert}_2=1$, the short-time Fourier transform provides the following analysis-synthesis formula:
\begin{align}\label{eq_rec}
f(t) = \int_{\mathbb{R}^{2d}} V_g f(x,\xi) \, g(t-x) e^{2 \pi i \xi  t} \,dx d\xi, \qquad t \in \mathbb{R}^d,
\end{align}
where the integral converges in quadratic mean.

\emph{Time-frequency concentration operators} are defined by restricting integration
in \eqref{eq_rec} to a compact time-frequency domain $\Omega \subset \mathbb{R}^{2d}$:
\begin{align}\label{eq_tfo}
\lo_\Omega f(t) = \int_{\Omega} V_g f(x,\xi) \, g(t-x) e^{2 \pi i \xi  t} \,dx d\xi, \qquad f \in L^2(\mathbb{R}^d).
\end{align}
Introduced in \cite{da88}, concentration operators for the short-time Fourier transform distinguish themselves from other forms of phase-space localization in that they treat the joint time-frequency variable $z=(x,\xi)$ as a single entity.

It is easy to see that $\lo_\Omega$ is positive semi-definite and contractive.  Heuristically $\lo_\Omega$ represents a projection onto the space of functions whose short-time Fourier transform is essentially localized on $\Omega$; the extent to which that intuition is correct is determined by the profile of the spectrum $\sigma(\lo_\Omega)$.

Weyl asymptotics for general time-frequency concentration operators involve dilations of a fixed domain $\Omega$ and go back to \cite{MR1310645}. In the refined version of \cite{MR1827072} they provide the following for each fixed $\delta \in (0,1)$:
\begin{align*}
\frac{\#\sett{\lambda \in \sigma(\lo_{R \cdot \Omega}): \lambda > \delta}}{\abs{R \cdot \Omega}} \longrightarrow 1,
\mbox{ as }R \longrightarrow +\infty,
\end{align*}
where $|E|$ denotes the Lebesgue measure of $E$, and eigenvalues are counted with multiplicity --- i.e., ${\#\sett{\lambda \in \sigma(\lo_{\Omega}): \lambda > \delta}} = \mathrm{trace}\big[1_{(\delta,\infty)}(\lo_\Omega)\big]$.

The number of intermediate eigenvalues $\#\sett{\lambda \in \sigma(\lo_{\Omega}): \delta < \lambda < 1- \delta}$ - dubbed \emph{plunge region} in \cite{da88} - essentially encodes the error of the approximation $\#\sett{\lambda \in \sigma(\lo_{\Omega}): \lambda > \delta} \approx \abs{\Omega}$ as it is the difference of
two such error terms.  For suitably general windows $g$ and domains $\Omega$, the following asymptotics were obtained in \cite{dMFeNo}:
\begin{align}\label{eq_A0}
0<c_{g, \delta, \Omega} \leq 
\frac{\#\sett{\lambda \in \sigma(\lo_{R \cdot \Omega}): \delta < \lambda < 1-\delta}}
{\mathcal{H}^{2d-1}\big(\partial [ R \cdot \Omega]\big) } \leq C_{g, \delta, \Omega}<\infty,
\end{align}
for all sufficiently large $R$ and sufficiently small $\delta>0$, depending on $\Omega$ and $g$. Here, $\mathcal{H}^{2d-1}$ is the $2d-1$ dimensional Hausdorff measure.
In fact, \cite{dMFeNo} derives similar asymptotics more generally for families of domains $\Omega$ where certain qualities related to curvature are kept uniform,
the family of dilations $\{R \cdot \Omega: R>0\}$ being one such example.\footnote{The quotation is not exact, as \cite{dMFeNo} uses the $2d$ dimensional measure of a neighborhood of the boundary instead of the perimeter; under the assumptions in \cite{dMFeNo} these are equivalent.}

Precise two-term {S}zeg\H{o} asymptotics for the eigenvalue counting function of a time-frequency concentration operator with Schwartz window $g$ and under increasing dilation of a $C^2$ domain $\Omega$ were derived in \cite{MR3433287}. The key quantity is
\begin{align}\label{eq_A1}
A_1(g,\partial\Omega,\delta) = C_d \cdot \int_{\partial\Omega}
\inf \Big\{\lambda \in \mathbb{R}: \int_{\{z \cdot n_u \,>\, \lambda\} } Wg(z) \,dz< \delta \Big\}
\,d\mathcal{H}^{2d-1}(u),
\end{align}
where $Wg$ is the \emph{Wigner distribution of $g$} \cite[Chapter 1]{folland} and $n_u$ is the outer unit normal of $\partial \Omega$ at $u$. The asymptotic expansion reads
\begin{align}\label{eq_A1r}
\#\sett{\lambda \in \sigma(\lo_{R \cdot \Omega}): \lambda > \delta} = 
\abs{R\cdot\Omega} + A_1(g,\partial\Omega,\delta) \cdot R^{2d-1} + o_{\delta,\Omega,g}(R^{2d-1}), \mbox{ as }R \longrightarrow +\infty.
\end{align}

\subsection{Robust versus non-robust error terms}\label{sec_quality}
The asymptotic expansion \eqref{eq_A1r} is \emph{non-robust with respect to $\delta$} in the sense that \eqref{eq_A1r} is only valid when $\delta$ is considered fixed with respect to $R$. Indeed, while the leading error term in \eqref{eq_A1r}, $A_1(g,\partial\Omega,\delta)$, depends explicitly on $\delta$, $g$ and $\Omega$, the error bound $o_{\delta,\Omega,g}(R^{2d-1})$ is not uniform on these parameters, \emph{which precludes applications where $\delta$ is allowed to vary with $R$}.

In contrast, \emph{upper bounds} for the error
$| \#\sett{\lambda \in \sigma(\lo_{\Omega}): \lambda > \delta} - \abs{\Omega} |$ that are \emph{threshold-robust} are possible
even for a fixed (non-dilated) domain. For example, if $\mathcal{H}^{2d-1}(\partial\Omega) < \infty$, comparing the first two moments of $\lo_\Omega$, as done for example in \cite[Proposition 3.3]{AbGrRo}, gives 
\begin{align}\label{eq_simple}
\big| \#\sett{\lambda \in \sigma(\lo_{\Omega}): \lambda > \delta} - \abs{\Omega} \big| \leq C_g \cdot \mathcal{H}^{2d-1}(\partial\Omega) \cdot \max\left\{\tfrac{1}{\delta}, \tfrac{1}{1-\delta}\right\}, \quad \! \delta \in (0,1),
\end{align}
where $C_g = \int_{\mathbb{R}^{2d}} |z| \abs{V_gg(z)}^2 \,dz$.

Threshold-robust spectral deviation bounds such as \eqref{eq_simple} enable applications in which $\delta$ is chosen as a function of the geometry of the domain $\Omega$. For example, Landau's method to study the discrepancy of sampling and interpolating sets
\cite{la67} relies on a formula analogous to \eqref{eq_simple}, whose most subtle applications require choosing $\delta$ proportional to a negative power of $\abs{\Omega}$; see, e.g. \cite[Proposition 1.3]{amro20}. Similarly, the quantification of the performance of randomized linear algebra methods ~\cite{MR2806637} applied to Toeplitz operators demands robust spectral bounds, and this has motivated the revisitation of the classical Landau-Widom asymptotics for the truncation of the Fourier transform \cite{israel15,KaRoDa,BoJaKa,Os}.

The purpose of this article is to provide a threshold-robust spectral deviation bound for time-frequency concentration operators \eqref{eq_tfo},
valid for a concrete (non-dilated) domain $\Omega$ and fully explicit on the threshold $\delta$, as in \eqref{eq_simple}, but with a more favorable dependence on $\delta$.

\subsection{Results}

We say that a compact domain $\Omega \subseteq \mathbb{R}^{2d}$ has \emph{regular boundary at scale} $\eta>0$ if there exists a constant $\kappa>0$ such that
\begin{align}\label{regbou}
\mathcal{H}^{2d-1}\big(\partial \Omega \cap B_{r}(z) \big)\geq \kappa \cdot r^{2d-1}, \qquad 0 < r \leq \eta, \quad z \in \partial\Omega.
\end{align}
In this case, the largest possible constant $\kappa$ is denoted
\begin{align}\label{eq_kappa}
\cu = \inf_{0 < r \leq \eta} \inf_{z \in \partial \Omega}
\frac{1}{r^{2d-1}} \cdot \mathcal{H}^{2d-1}\big(\partial \Omega \cap B_{r}(z) \big).
\end{align}
This condition means that $\partial \Omega$ satisfies the lower estimate in the definition of \emph{Ahlfors regularity} (see \cite[Definition~I.1.13]{DaSe}). Aside from Lipschitz domains, which satisfy \eqref{regbou}, boundary regularity allows for ridges --- for example, sets of the form $[0,1]^n \times E\subseteq \R^{n+2}$ where $\partial E\subseteq \R^2$ is compact and connected.

In this article we prove the following.
\begin{theorem}[Threshold-robust spectral bounds, non-dilated domains]\label{th_intro_1}
Let $\Omega\subseteq\R^{2d}$ be a compact set with regular boundary at scale $\eta >0$. Let $g \in L^2(\mathbb{R}^d)$ satisfy $\lVert g \rVert_2=1$ and the following Gelfand-Shilov-type condition with parameter $\beta \geq 1/2$: there exist $C_g,A>0$ such that for every $n \in \mathbb{N}_0$: 
\begin{align}\label{eq_gs}
\big| \, V_g g(z) \, \big| \leq C_g A^n n!^\beta (1+|z|)^{-n}, \qquad z \in \mathbb{R}^{2d}.
\end{align}

For $\delta\in(0,1)$ set $\tau=\max\left\{\tfrac{1}{\delta}, \tfrac{1}{1-\delta}\right\}$. Then
\begin{align}\label{eq_t1}
\big| \#\sett{\lambda \in \sigma(\lo_{ \Omega}): \lambda > \delta} - \abs{\Omega} \big|
\leq C'_g \cdot
\mathcal{H}^{2d-1}(\partial \Omega) \cdot (\log \tau)^{\beta } \left(1+\frac{(\log \tau)^{(2d-1)\beta}}{\eta^{2d-1} }\right)
 \frac{	\log\big[\log(\tau)+1\big]}{\cu },
\end{align}
where $C'_g  = C_g^{1/2} \cdot A^{3d+2} \cdot C_d^\beta$ and $C_d$ depends only on $d$. 
\end{theorem}
\begin{rem}
Condition \eqref{eq_gs} is satisfied whenever $g$ belongs to the \emph{Gelfand-Shilov class} $\mathcal{S}^{\beta,\beta}$, i.e., if there exist constants $a,b,C>0$ such that the following decay and smoothness conditions hold:
\begin{align}\label{eq_gsc}
|g(x)| \leq C e^{-a |x|^{1/\beta}}, \qquad |\hat{g}(\xi)| \leq C e^{-b |\xi|^{1/\beta}}, \qquad x,\xi \in \mathbb{R}^d;
\end{align}
see \cite{MR3469849, ChChKi}, \cite[Theorem~3.2]{MR1732755}, \cite[Corollary 3.11 and Proposition 3.12]{MR2027858},
\cite[Proposition 3.3]{MR1916857} and \cite[Theorem~3.1]{YaOb}. Conversely,
\eqref{eq_gs} implies that $g\in \mathcal{S}^{\beta',\beta'}$ for every $\beta'>\beta$. 
\end{rem}

The restriction to $\beta \geq 1/2$ in \eqref{eq_gs} is natural, as a version of Hardy's uncertainty principle \cite{MR1802767} precludes the case $\beta < 1/2$.
As we show in Section \ref{sec_disc}, the estimate in Theorem \ref{th_intro_1} is \emph{sharp in the spectral threshold $\delta$ up to $\log\log$ factors}. In particular, the exponents $1$ and $2d-1$ on $(\log \tau)^\beta$ in \eqref{eq_t1} cannot be improved.

\subsection{The dilation regime}\label{sec_dil}
To better appreciate Theorem \ref{th_intro_1}, let us specialize it to the dilation regime. Note that if $\Omega$ has regular boundary at scale $\eta$ with parameter $\cu$, then $R\cdot\Omega$ has regular boundary at scale $R\cdot \eta$ with parameter $\cu$. Thus, Theorem \ref{th_intro_1} implies the following.

\begin{coro}[Threshold-robust spectral bounds, the dilation regime]\label{coro_dil}
Let $\Omega\subseteq\R^{2d}$ and $g \in L^2(\mathbb{R}^d)$ satisfy the conditions of Theorem~\ref{th_intro_1}. Then, with the same notation, for all $R>0$,
\begin{align}\label{eq_good}
\begin{aligned}
	\big|\#&\sett{\lambda \in \sigma(\lo_{R \cdot \Omega}): \lambda > \delta} -
\abs{R\cdot\Omega}\big| 
\\ &\le C'_g \cdot
\mathcal{H}^{2d-1}\big(\partial [R\cdot\Omega]\big) \cdot (\log \tau)^{\beta } \left(1+\frac{(\log \tau)^{(2d-1)\beta}}{(R\cdot\eta)^{2d-1} }\right) 
\frac{	\log\big[\log(\tau)+1\big]}{\cu }.
\end{aligned}
\end{align}
\end{coro}
Let us compare Corollary \ref{coro_dil} with the spectral deviation estimate that follows from the asymptotics in \cite{MR3433287}.
Suppose for concreteness that $g$ satisfies the Gelfand-Shilov condition \eqref{eq_gsc}. Then, as shown in \cite{MR1732755}, there are constants $c,C'>0$ such that
\[|Wg(z)|\le C' e^{-c|z|^{1/\beta}}.\]
Inspection of \eqref{eq_A1} --- say, for $\delta < 1/2$ --- gives
\begin{align}\label{eq_ol1a}
|A_1(g,\partial\Omega,\delta)| \leq C_g \cdot \mathcal{H}^{2d-1}(\partial\Omega) \cdot [\log(1/\delta)]^\beta.
\end{align}
The asymptotic expansion \eqref{eq_A1r} now shows that there exist $R_0=R_0(\Omega,\delta,g)$ such that
\begin{align}\label{eq_bad}
\big|\#\sett{\lambda \in \sigma(\lo_{R \cdot \Omega}): \lambda > \delta} -
\abs{R\cdot\Omega}\big| 
\le C_{g,\Omega} \cdot
R^{2d-1} \cdot [\log (1/\delta)]^{\beta },
\qquad R \geq R_0(\Omega,g,\delta).
\end{align}
No information on the dependence of $R_0$ on $\delta$ is available and applications of \eqref{eq_bad} where $R$ is taken as a function of $\delta$ (or vice-versa) are thus not possible.

In contrast, \eqref{eq_good} --- still with $\delta<1/2$ --- provides the threshold-robust bound:
\begin{align}\label{eq_better}
\begin{aligned}
&\big|\#\sett{\lambda \in \sigma(\lo_{R \cdot \Omega}): \lambda > \delta} -
\abs{R\cdot\Omega}\big| 
\\
&\qquad\le C_{g,\Omega} \cdot
R^{2d-1} \cdot [\log (1/\delta)]^{\beta } \cdot \log\big[\log(1/\delta)+1\big], \qquad
R \geq \frac{[\log (1/\delta)]^{\beta}}{\eta},
\end{aligned}
\end{align}
and allows, for example, for
$R \rightarrow \infty$, while $\delta=R^{-s}$, with $s>0$.

While the dependence of \eqref{eq_better} on $\delta$ almost matches the one in the less quantitative estimate \eqref{eq_bad}, one may wonder if the comparison is fair, since \eqref{eq_bad} was derived by means of the upper bound \eqref{eq_ol1a}.
As it turns out, the dependence of \eqref{eq_better} in $\delta$ is indeed best possible up to $\log\log$ factors. This almost sharpness of Corollary \ref{coro_dil}  is discussed in Section \ref{sec_disc} for different regimes of $R$ and $\delta$ by testing the bound on explicit examples.

\subsection{Polynomial time-frequency decay}
We also provide a variant of Theorem \ref{th_intro_1} for window functions with only polynomial time-frequency decay. 

\begin{theorem}\label{th_2}
	Let $\Omega\subseteq\R^{2d}$ be a compact set with regular boundary at scale $\eta \in(0,1]$. 
	Let $g \in L^2(\mathbb{R}^d)$ be such that $\lVert g \rVert_2=1$,
	and, for some $s\ge 1$,
	\begin{align}\label{eq_mod}
	C_g :=  \int_{\mathbb{R}^{2d}} (1+|z|)^{s} |V_g g(z)|^2 \,dz <\infty.
	\end{align}
	
	For $\delta\in(0,1)$ set $\tau=\max\left\{\tfrac{1}{\delta}, \tfrac{1}{1-\delta}\right\}$. Then
	\begin{align}\label{eq_tpoli}
	\big| \#\sett{\lambda \in \sigma(\lo_{ \Omega}): \lambda > \delta} - \abs{\Omega} \big|
	\le C'_g
	\cdot  \mathcal{H}^{2d-1}(\partial \Omega) \cdot 
	\tau^{\frac{2d}{2d+s-1}} \cdot \left(\frac{\log (C_g \tau)}{\cu \cdot \eta^{2d-1} }\right)^{\frac{s-1}{2d+s-1}} ,
	\end{align}
	where $C'_g = C_{d} \cdot C_g^{\frac{2d}{2d+s-1}}$, and $C_d$ depends only on $d$.
\end{theorem}
Note that for $s=1$, \eqref{eq_tpoli} recovers \eqref{eq_simple} (with a larger constant $C'_g$). As $s$ grows, the dependence on $\tau$ becomes milder while the geometric constant $\cu$ becomes more relevant.
\begin{rem}\label{rem_m1s}
The hypothesis
\eqref{eq_mod} is satisfied, for example, if $g$ belongs to the modulation space
$M^1_{s}(\mathbb{R}^d)$ \cite{benyimodulation}; see Section \ref{sec_m1s}.
\end{rem}

\subsection{Some consequences}
We first note the following.
\begin{rem}
	Under the hypotheses of Theorem \ref{th_intro_1} or \ref{th_2}, the error bounds \eqref{eq_t1} or \eqref{eq_tpoli} also apply to the plunge region
	\begin{align*}
	\#\sett{\lambda: \delta < \lambda \leq 1-\delta} = 
	\left(
	\#\sett{\lambda: \lambda > \delta} - \abs{\Omega}
	\right)
	-
	\left(\#\sett{\lambda: \lambda > 1-\delta} - \abs{\Omega}\right).
	\end{align*}
\end{rem}

Second, spectral deviation bounds can of course be reformulated as asymptotics for individual eigenvalues. We now spell out the details.

We let $\{\lambda_k: k \geq 1\}$ be the decreasing rearrangement of the eigenvalues of $\lo_{ \Omega}$ counting multiplicities.
According to \eqref{eq_t1} and \eqref{eq_tpoli}, the eigenvalues are expected to transition from about 1 to about 0 around the index
$A_\Omega := \ceil{\abs{\Omega}}$ (smallest integer $\geq \abs{\Omega}$). Theorem \ref{th_intro_1} yields the following estimates.

\begin{coro}\label{coro_1}
	Let $\Omega\subseteq\R^{2d}$ and $g \in L^2(\mathbb{R}^d)$ satisfy the conditions of Theorem~\ref{th_intro_1} with $\eta\in(0,1]$. Using the theorem's notation write
	\[\gamma=
	\frac{2	C'_g}
		{\cu \cdot \eta^{2d-1}} \cdot \mathcal{H}^{2d-1}(\partial \Omega).\]
	Then the following hold.
	\begin{enumerate}[label=(\roman*),ref=(\roman*)]
		\item\label{coro_1i} For every $k=A_\Omega+\gamma h\in\N$ with $h\ge 1$,
		\begin{align}
		\lambda_k &\le e^{-\left(\frac{h}{e(1+\log h)}\right)^{1/{2d\beta}}}. \label{eqeig1}
		\end{align}
		
		\item\label{coro_1ii} For every $k=A_\Omega-\gamma h\in\N$ with $h\ge 1$,
		\begin{align}
		\lambda_k &\ge 1- e^{-\left(\frac{h}{e(1+\log h)}\right)^{1/{2d\beta}}}. \label{eqeig2}
		\end{align}
	\end{enumerate}
\end{coro}

Similarly, for windows with polynomial time-frequency decay, Theorem \ref{th_2} gives the following.

\begin{coro}\label{coro_2}
	Let $\Omega\subseteq\R^{2d}$ and $g \in L^2(\mathbb{R}^d)$ satisfy the conditions of Theorem~\ref{th_2}. Using the theorem's notation write
	\begin{align}\label{eq_gamma}
	\gamma=
		C_g' \cdot \left(\cu \cdot \eta^{2d-1}\right)^{-\frac{s-1}{2d+s-1}} \cdot
	\mathcal{H}^{2d-1}(\partial \Omega).
	\end{align}
	Then the following hold.
	\begin{enumerate}[label=(\roman*),ref=(\roman*)]
		\item\label{coro_2i} For every $k=A_\Omega+\gamma h\in\N$ with $h\ge 1$,
		\begin{align}
		\lambda_k &\le  e^{\left(\frac{2d+s-1}{2d}\right)^2} \cdot h^{-\frac{2d+s-1}{2d}}\cdot (1+\log (C_g h))^{\frac{s-1}{2d}}. \label{eqeig3}
		\end{align}
		\item\label{coro_2ii} For every $k=A_\Omega-\gamma h\in\N$ with $h\ge 1$,
		\begin{align}
		\lambda_k &\ge 1- 
		 e^{\left(\frac{2d+s-1}{2d}\right)^2} \cdot h^{-\frac{2d+s-1}{2d}} \cdot
		(1+\log (C_g h))^{\frac{s-1}{2d}}. \label{eqeig4}
		\end{align}
	\end{enumerate}
\end{coro}

\subsection{Related literature and organization}
Eigenvalue asymptotics related to the truncation of the Fourier transform go back at least to \cite{MR593228}. These describe precisely the limit behavior of the eigenvalue counting function as the truncation domain grows. The description is however only valid for a fixed eigenvalue threshold, and may be therefore insufficient in many applications --- see Section \ref{sec_quality}. The need for corresponding threshold-robust spectral deviation bounds has motivated substantial research \cite{israel15,KaRoDa,BoJaKa,Os}. While this article shares a similar goal, with the exception of \cite{israel15}, the existing literature on threshold-robust spectral bounds for the truncation of the Fourier transform exploits specific properties of that context, such as a connection with a Sturm–Liouville equation, that are not available in our setting.

Time-frequency concentration operators for the short-time Fourier transform are also called localization operators \cite{da88} or (generalized) anti-Wick operators \cite{MR0291839, MR2047590}, specially when considered, in greater generality, with respect to symbols, so that the truncated integral in \eqref{eq_tfo} is replaced by a weighted integral. A remarkable special case occurs when $g$ is the one variable Gaussian window $g(t)=2^{1/4} e^{-\pi t^2}$ and $\Omega$ is a disk, as $\lo_\Omega$ is then diagonal in the Hermite basis, and the spectrum can be computed explicitly \cite{da88}. We use this example to test the sharpness of our estimates. 

In the case of time-frequency concentration operators, the benchmark result for spectral asymptotics is in \cite{MR3433287}, partially building on \cite{MR3076415}. As explained in Section \ref{sec_quality}, when applied to produce spectral deviation bounds (that is, bounds on $|\#\sett{\lambda \in \sigma(\lo_{R\cdot\Omega}): \lambda > \delta} - \abs{R\cdot\Omega}|$), the asymptotics of \cite{MR3433287} deliver conclusions that are ineffective in applications where $R$ and $\delta$ vary together. In contrast, when particularized to the dilation regime, Theorems \ref{th_intro_1} and \ref{th_2} yield threshold-robust spectral deviation bounds, 
which do allow for $R$ and $\delta$ to vary together. In addition, Theorems \ref{th_intro_1} and \ref{th_2} are applicable beyond the dilation regime.

The lacking quantitative information in the error terms of \eqref{eq_A1r} cannot be found by simply reinspecting the proofs in \cite{MR3433287}, which are conceived with the different objective of capturing the precise value of the first asymptotic term \eqref{eq_A1}. Indeed, inspection of the proofs in \cite{MR3433287} yields bounds on the term $o_{\delta,\Omega,g}$ on the right-hand side of \eqref{eq_A1r} that are exponentially worse than \eqref{eq_good}, and which may overtake the first error term \eqref{eq_A1r} if $\delta$ is allowed to fluctuate, even moderately, with $R$. Similarly, inspection of the proofs in \cite{dMFeNo} yield constants in \eqref{eq_A0} of the order $O\big(1/\sqrt\delta\big)$.

From the point of view of pseudo-differential theory, time-frequency concentration operators are Weyl quantizations of symbols of the form $\sigma = Wg * 1_\Omega$, where $Wg$ is the Wigner distribution of the window function $g$ and $*$ is the standard convolution. The results in \cite{MR3433287} apply more generally to operators with Weyl symbol $\sigma= W * [ a \cdot 1_\Omega ]$ where $a$ and $W$ are suitably smooth functions. In contrast, our proofs rely on a more specific characteristic of time-frequency concentration operators, namely that they are unitarily equivalent to Toeplitz operators on the range of the short-time Fourier transform \cite{MR1882695}.
Such spaces are reproducing kernel subspaces of $L^2(\mathbb{R}^{2d})$ and share certain properties with the Fock space of square integrable analytic functions on $\mathbb{C}^d$ with respect to the Gaussian measure. The core of our proof is to estimate Schatten $p$-norms of Hankel operators, with a value of $p$ that is optimized as a function of the spectral parameter $\delta$. We draw mainly on arguments from \cite{XiZhe}, \cite{Is}, developed to study Hankel operators on Fock spaces, and combine them with methods from geometric measure theory \cite{evga92, Ca}. While the $p$-dependence of Schatten norms is normally left unquantified \cite{XiZhe}, \cite{Is}, \cite{MR3182964}, in our case capturing that dependence is essential. We also greatly benefited from reading \cite{MR3636061}.

The proofs of Theorem \ref{th_intro_1} and the auxiliary Proposition \ref{prop1} rely on well-known methods in spectral theory. Our contribution is in the quantification of certain aspects by means of tools from geometric measure theory. Remarkably, our approach leads to almost-sharp estimates in the spectral parameter; see Section \ref{sec_disc}.

The article is organized as follows. Section~\ref{sec_pre} provides background results. In Section~\ref{sec_hankel} we derive Schatten norm estimates for Hankel operators on the range of the short-time Fourier transform. These are applied in Section \ref{sec_main} to prove Theorems \ref{th_intro_1} and \ref{th_2},
and in turn Corollaries \ref{coro_1} and \ref{coro_2} in Section \ref{sec_eigasy}. In Section \ref{sec_disc} we test the sharpness of our main estimates. Remark \ref{rem_m1s} is proved in Section \ref{sec_m1s}. Some of our proofs require revisiting standard arguments while carefully tracking dependencies on certain parameters, and are therefore given in detail. Detailed proofs of Corollaries \ref{coro_1} and \ref{coro_2} are also included for completeness.

\section{Preliminaries}\label{sec_pre}
\subsection{Notation}
We always let $|z|$ denote the Euclidean norm of a vector $z \in \mathbb{R}^{2d}$.
The differential of the Lebesgue measure is denoted $dz$. We write $1_E$ for the indicator function of a set $E \subseteq \mathbb{R}^{2d}$
and $E^C$ for its complement $E^C = \mathbb{R}^{2d} \setminus E$.
We write $a\lesssim b$ for two non-negative quantities $a,b$ if there exists a constant $C_d$ depending only on the dimension $d$ such that $a\le C_d\, b$ holds. $L^p$ norms with respect to polynomial weights are denoted as
\begin{align*}
\big\lVert (1+|z|)^{s} F(z) \big\rVert_{L^p(dz)}^p = 
\int_{\mathbb{R}^{2d}} (1+|z|)^{ps} |F(z)|^p \, dz,
\end{align*}
or $\lVert (1+|z|)^{s} F \rVert^p_p$ for short. 

For $F,G\in L^2(\R^{2d})$ we write $F\natural G$ for the \emph{twisted convolution}:
\[F\natural G(z)=\int_{\R^{2d}} F(z') G(z-z') e^{\pi i (x\xi'-x'\xi)}\, dz', \qquad z=(x,\xi),\, z'=(x',\xi') \in \mathbb{R}^d \times \mathbb{R}^d.\]
We will use the following consequence of \emph{Pool's Theorem} \cite{MR204049}:
\begin{align}\label{eq_tc}
\|F\natural G\|_2\le \|F\|_2\|G\|_2;
\end{align}
see also \cite[Chapter~2]{folland} and \cite[Corollary~14.6.2]{grbook}.
\subsection{Operators}
The spectrum of an operator $A$ on a Hilbert space is denoted $\sigma(A)$, and its Schatten (quasi)norm is given by
\begin{align*}
\|A\|_p^p = \mathrm{trace}(|A|^p)=\mathrm{trace}((A^*A)^{p/2}), \qquad p>0.
\end{align*}
The eigenvalue counting function of a self-adjoint operator $A$ is
\begin{align*}
\# \{ \lambda \in \sigma(A): \sigma > \delta \} = \mathrm{trace}\big[ 1_{(\delta,\infty)}(A) \big].
\end{align*}
In general, all expressions involving $\sigma(A)$ are to be understood to include mutiplicities.

We will use the following elementary lemma, which follows from a convexity argument. (See, e.g., \cite[Proposition~6.3.3]{zhu2007operator} for a proof.)
\begin{lemma}\label{lemma_p}
	Let $A$ be a positive operator on a Hilbert space $H$. Then
	\begin{align*}
	\ip {A^p x}{x} \leq \ip{Ax}{x}^p,
	\qquad \lVert x \rVert=1, \qquad 0<p\le1.
	\end{align*}
\end{lemma}

\subsection{Perimeters and level sets}
The topological boundary of a set $\Omega \subseteq \mathbb{R}^{2d}$ is 
denoted $\partial \Omega$. Besides working with $\Omega$ compact, we will assume without further mention that 
$\mathcal{H}^{2d-1}(\partial \Omega) < \infty$ as the bounds we prove are otherwise trivial. Since $\Omega$ has finite measure, $\Omega$ has \emph{finite perimeter}, i.e., its indicator function has bounded variation; see, e.g., \cite[Chapter 5]{evga92}. Indeed, $\Omega$ has finite perimeter if and only if
$\mathcal{H}^{2d-1}(\partial_* \Omega) < \infty$, where $\partial_*\Omega \subset \partial\Omega$ is the \emph{measure theoretic boundary}, consisting
of points with positive density both on $\Omega$ and $\Omega^C$
\cite[Theorem 1, Section 5.11]{evga92}. Similarly, we assume without further mention that $\mathcal{H}^{2d-1}(\partial \Omega)>0$, since, otherwise, by the isoperimetric inequality, $\abs{\Omega}=0$.

We will need the following regularization lemma.
\begin{lemma}
	\label{lemma_var}
	Let $\Omega \subseteq \mathbb{R}^{2d}$ be a compact set, let
	$E=\Omega$ or $\Omega^C$ and let $\varphi \in L^1(\mathbb{R}^{2d})$.  Then
	\begin{align*}
	\Big\lVert{1_E *\varphi- \big(\smallint \varphi\big) \cdot 1_E}\Big\rVert_{L^1(\mathbb{R}^{2d})} \leq
	\mathcal{H}^{2d-1}(\partial \Omega) \cdot \int_{\mathbb{R}^{2d}} \abs{z} \abs{\varphi(z)} dz.
	\end{align*}
\end{lemma}
\begin{proof}
When $E=\Omega$ a simple proof can be found in \cite[Lemma~3.2]{AbGrRo}. That reference actually provides a better bound in terms of
the measure theoretic boundary: $\mathrm{var}(1_\Omega) = \mathcal{H}^{2d-1}(\partial_* \Omega) \leq \mathcal{H}^{2d-1}(\partial \Omega)$. If $E=\Omega^C$, we use that $1_E *\varphi- \big(\smallint \varphi\big) \cdot 1_E=
 \big(\smallint \varphi\big) \cdot 1_{E^C} - 1_{E^C} *\varphi$.
\end{proof}

The following bound for level sets of distances will be a key technical tool. It is a special case of \cite[Theorems 5 and 6]{Ca}; see \cite{Kr} for related results.
\begin{prop}\label{prop_ca}
Let $\Omega \subseteq \mathbb{R}^{2d}$ be a compact set with regular boundary at scale $\eta>0$. Then
\begin{align*}
\mathcal{H}^{2d-1} \big(\{z:\ d(z,\partial \Omega)=r\} \big) \leq  \frac{C_d}{\cu}  \cdot \mathcal{H}^{2d-1}(\partial \Omega) \cdot \Big(1+\frac{r}{\eta}\Big)^{2d-1},
\end{align*}
for almost every $r>0$, where $C_d$ is a constant that depends on $d$. In addition,
$| \nabla d(z,\partial\Omega) | = 1$, for almost every $z \in \mathbb{R}^{2d}$.
\end{prop}
Note that if $\Omega$ has regular boundary at scale $\eta>0$
(and assuming as we do that $\mathcal{H}^{2d-1}(\partial \Omega)>0$), it follows by differentiation around a point of positive $\mathcal{H}^{2d-1}$-density that \eqref{eq_kappa} satisfies
\begin{align}\label{eq_kappa_2}
\cu \leq c_d,
\end{align}
for a dimensional constant $c_d$.

\subsection{Toeplitz and Hankel operators on the range of the STFT}
For a normalized window function $g \in L^2(\mathbb{R}^{d})$,
$\lVert{g}\rVert_2=1$, the short-time Fourier transform
defines an isometry $V_g:L^2(\mathbb{R}^d) \to L^2(\mathbb{R}^{2d})$ \cite[Chapter 3]{grbook}. 

The range of the short-time Fourier transform with window $g$ is the (closed) subspace
\begin{align*}
\hh = \left\{V_g f: f \in L^2(\mathbb{R}^d) \right\} \subset L^2(\mathbb{R}^{2d}),
\end{align*}
and the orthogonal projection $P: L^2(\mathbb{R}^{2d}) \to \hh$ is the integral operator
\begin{align*}
P F (z) = \int_{\mathbb{R}^{2d}} K(z,w) F(w) dw,
\end{align*}
where $K$ is the \emph{reproducing kernel} 
\begin{align*}
K(z,w) = K_w(z) = V_gg (z-w) e^{2\pi i (\xi'-\xi) x'}, \qquad z=(x,\xi), w=(x',\xi') \in \mathbb{R}^d \times \mathbb{R}^d.
\end{align*}
In particular
$\lVert K_w \rVert_{L^2(\mathbb{R}^{2d})} 
=1$.
Since $P$ is a projection, $K$ satisfies the \emph{reproducing formula}
\begin{align}\label{eq_rep}
K(z,w) = \int_{\mathbb{R}^{2d}} K(z,z') K (z',w) \,dz', \qquad z,w \in \mathbb{R}^{2d}.
\end{align}
If  $\{e_n\}_n$ is an orthonormal basis of $\hh$, then $K$ can be expanded as $K(z,w)=\sum_n e_n(z)\overline{e_n(w)}$.

In terms of the adjoint short-time Fourier transform $V^*_g: L^2(\mathbb{R}^{2d}) \to L^2(\mathbb{R}^d)$, $V^*_g F(t) = \int_{\mathbb{R}^{2d}} F(x,\xi) g(t-x) e^{2 \pi i \xi t} \,dx d\xi$, the time-frequency concentration operator $L_\Omega$ is
\begin{align*}
L_\Omega f = V^*_g \big[ 1_\Omega \cdot V_g f \big].
\end{align*}
This shows that time-frequency concentration operators \eqref{eq_tfo} are unitarily equivalent to contractions of certain Toeplitz operators on the range of the short-time Fourier transform. Indeed, given a compact domain $\Omega \subseteq \mathbb{R}^{2d}$, we define the
\emph{Gabor-Toeplitz} operator $T_\Omega:L^2(\mathbb{R}^{2d}) \to L^2(\mathbb{R}^{2d})$ by
\begin{align*}
T_\Omega F = P \big( 1_\Omega \cdot PF \big), \qquad F \in L^2(\mathbb{R}^{2d}).
\end{align*}
As $P=V_g V^*_g$, it follows that, with respect to the decomposition $L^2(\mathbb{R}^{2d}) = \hh \bigoplus \hh^\perp$,
\begin{align}\label{eq_conj}
T_\Omega F = \begin{bmatrix}
V_g  L_\Omega  V_g^* F  & 0\\
0 & 0
\end{bmatrix},
\end{align}
see, e.g., \cite{dMFeNo, MR1882695, MR2452833}.
Thus, the spectrum of $T_\Omega$ and $\lo_\Omega$ coincide except for the multiplicity of the eigenvalue $\lambda=0$.
A direct calculation shows that
\begin{align}\label{eq_traces}
\tr(\lo_\Omega)=|\Omega|, \qquad \text{and} \qquad \tr(\lo_\Omega^2)=\int_\Omega \int_\Omega |V_g g(z-z')|^2 \, dz dz',
\end{align}
see, e.g., \cite[Lemma~2.1]{AbGrRo}.

Our proofs are based on the analysis of the \emph{Hankel operator}
$H_\Omega = (1-P) 1_\Omega P$, i.e.,
\begin{align*}
H_\Omega F = 1_\Omega \cdot P F - P \big( 1_\Omega \cdot PF \big),
\qquad F \in L^2(\mathbb{R}^{2d}).
\end{align*}
We sometimes drop the dependence on $g$ and $\Omega$, and simply write $T$ and $H$. Notice that $H$ and $T$ are related by,
\begin{align}\label{eq_ht}
H^*H=T-T^2.
\end{align}

\section{Estimates for Hankel operators on the range of the STFT}\label{sec_hankel}
\begin{prop}\label{prop1}
	Let $\Omega\subseteq\R^{2d}$ be a compact set with regular boundary at scale $\eta >0$. Let $p \in(0,2]$ and $\alpha>0$.
	Then
	\begin{align}\label{eq_p21}
		\|H\|_p^p\leq C_d \cdot 
		\mathcal{H}^{2d-1}(\partial \Omega) \cdot
		\frac{ \left\|(1+|z|\eta^{-1})^{(2d-1)(2-p)/2p} (1+|z|)^{(1+\alpha)(2-p)/2p+1/2}V_g g\right\|_2^{p}}{(\cu \cdot \alpha)^{1-p/2}},
	\end{align}
	where $C_d$ is a constant that only depends on the dimension $d$.
	
	In particular, if $\eta \in(0,1]$,
	\begin{align}\label{eq_p2}
	\|H\|_p^p\leq C_d \cdot 
	\mathcal{H}^{2d-1}(\partial \Omega) \cdot
	\frac{ \left\|(1+|z|)^{(2d+\alpha)(2-p)/2p+1/2}V_g g\right\|_2^{p}}{(\cu \cdot \eta^{2d-1} \cdot \alpha)^{1-p/2}}.
	\end{align}
\end{prop}

\begin{proof}
\noindent {\bf Step 1}. We first argue as in the proof of \cite[Theorem~1.1]{XiZhe}. 
	
	Denote $A=(H^*H)^{p/2}$. Since $A|_{\hh^\perp}=0$, for an orthonormal basis $\{e_n\}_n$ of $\hh$ we have
	\begin{align*}
	\|H\|_p^p&=\tr(A)=\sum_n \langle A e_n, e_n\rangle=\sum_n \int_{\R^{2d}} Ae_n(z) \overline{e_n(z)} \, dz =\sum_n \int_{\R^{2d}} PAe_n(z) \overline{e_n(z)} \, dz
	\\ &=\sum_n \int_{\R^{2d}}\langle Ae_n,K_z\rangle  \overline{e_n(z)} \, dz
	= \int_{\R^{2d}}\big\langle A\sum_n e_n \overline{e_n(z)},K_z\big\rangle   \, dz
= \int_{\R^{2d}}\langle AK_z,K_z\rangle   \, dz
	\\ &=\int_{\R^{2d}}\langle (H^*H)^{p/2}K_z,K_z\rangle   \, dz.
	\end{align*}
	By Lemma \ref{lemma_p}, we have $\langle (H^*H)^{p/2}K_z,K_z\rangle\le \langle (H^*H)
	K_z,K_z\rangle ^{p/2}$ since $p/2\le 1$, and, consequently,
	\begin{align}\label{eq1}
	\|H\|_p^p&\le \int_{\R^{2d}}\langle H^*HK_z,K_z\rangle ^{p/2}  \, dz=\int_{\R^{2d}}\|HK_z\|_2 ^{p}  \, dz.
	\end{align}
	Let us describe the function $HK_z$:
	\begin{align*}
	HK_z(w)&= (1-P) \big[ \rchi_\Omega \cdot  PK_z \big](w)=(1-P) \big[ \rchi_\Omega \cdot K_z\big](w)
	\\ &=\rchi_\Omega(w)K(w,z)-\int_{\R^{2d}}\rchi_\Omega(z')K(z',z) K(w,z') \, dz'
	\\ &=\int_{\R^{2d}}(\rchi_\Omega(w)-\rchi_\Omega(z'))K (w,z')K(z',z) \, dz'
	\\ &=\int_{\R^{2d}}(\rchi_\Omega(w)\rchi_{\Omega^c}(z')-\rchi_{\Omega^c}(w)\rchi_\Omega(z'))K (w,z')K(z',z) \, dz'.
	\end{align*}
	So,
	\begin{align*}
	\|HK_z\|_2\le \, &\Big(\int_{\Omega} \Big|\int_{\Omega^c} K (w,z')K(z',z) \, dz'\Big|^{2}\, dw\Big)^{1/2}
	\\ &+\Big(\int_{\Omega^c} \Big|\int_{\Omega} K (w,z')K(z',z) \, dz'\Big|^{2}\, dw\Big)^{1/2}.
	\end{align*}
	We combine the last estimate with \eqref{eq1}, and use the fact that $(a+b)^{p}\le 2(a^p+b^p)$ for $a,b\ge 0$, $0<p\le 2$, to split the integrals similarly to \cite[Proof of Lemma~4.4]{dMFeNo}:
	\begin{align*}
	\|H\|_p^p\lesssim \, &\int_{\R^{2d}} \Big(\int_{\Omega} \Big|\int_{\Omega^c} K (w,z')K(z',z) \, dz'\Big|^{2}\, dw\Big)^{p/2}  \, dz
	\\ &+\int_{\R^{2d}} \Big(\int_{\Omega^c} \Big|\int_{\Omega} K (w,z')K(z',z) \, dz'\Big|^{2}\, dw\Big)^{p/2}  \, dz
	\\ =\, & \int_{\Omega} \Big(\int_{\Omega} \Big|\int_{\Omega^c} K (w,z')K(z',z) \, dz'\Big|^{2}\, dw\Big)^{p/2}  \, dz
	\\ &+ \int_{\Omega^c} \Big(\int_{\Omega} \Big|\int_{\Omega^c} K (w,z')K(z',z) \, dz'\Big|^{2}\, dw\Big)^{p/2}  \, dz
	\\ &+\int_{\Omega} \Big(\int_{\Omega^c} \Big|\int_{\Omega} K (w,z')K(z',z) \, dz'\Big|^{2}\, dw\Big)^{p/2}  \, dz
	\\ &+\int_{\Omega^c} \Big(\int_{\Omega^c} \Big|\int_{\Omega} K (w,z')K(z',z) \, dz'\Big|^{2}\, dw\Big)^{p/2}  \, dz.
	\end{align*}
		Hence,
	\begin{align}\label{eq_b1}
	\|H\|_p^p\lesssim \max \left\{I_j(E): \ j=1,2, E=\Omega, \Omega^C	
	\right\},
	\end{align}
	where
	\begin{align*}
	I_1(E) &= \int_{E} \Big(\int_{E} \Big|\int_{E^c} K (w,z')K(z',z) \, dz'\Big|^{2}\, dw\Big)^{p/2}  \, dz,
	\\I_2(E) &= \int_{E} \Big(\int_{E^c} \Big|\int_{E} K (w,z')K(z',z) \, dz'\Big|^{2}\, dw\Big)^{p/2}  \, dz.
	\end{align*}
	
	\noindent {\bf Step 2}. We let $E$ denote either $\Omega$ or $\Omega^C$; in both cases $\partial E=\partial \Omega$. 
	We write for short $I_1=I_1(E)$, $I_2=I_2(E)$, and introduce the quantity:
	\begin{align*}
 I_3 = I_3(E) &= \int_{E} \Big(\int_{E^c} | V_g g(z'-z)|^2 \, dz'\Big)^{p/2}  \, dz.
	\end{align*}
	Let us show that
	\begin{align}\label{eq_b2}
	I_1\le   I_3, \mbox{ and } I_2\le 4 I_3.
	\end{align}

	For $I_1$ define $F\in L^2(\R^{2d})$ by
	\[F(x,\xi)= V_g g (x,\xi) e^{\pi i x \xi},\]
	and notice that for $z=(x,\xi)$ and $z'=(x',\xi')$ we have
	\begin{align*}
		K(z',z) &=  V_g g (z'-z)  e^{2\pi i (\xi-\xi')x} 
		=  F (z'-z) e^{-\pi i (x'-x)(\xi'-\xi)}  e^{2\pi i (\xi-\xi')x}
		\\ &
		= F (z'-z) e^{\pi i (-x'\xi'+x'\xi-x\xi'+x\xi)}.
	\end{align*}
	Denoting $w=(x'',\xi'')$ we get
	\begin{align*}
		K(w,z')K(z',z) &=\hphantom{:} F (w-z') \cdot F (z'-z) \cdot e^{\pi i (-x''\xi'' + x'' \xi' -x'\xi'' +x'\xi' -x'\xi'+x'\xi-x\xi'+x\xi)}
		\\ &=\hphantom{:} F(w-z') \cdot F (z'-z) e^{\pi i (x'\xi-x\xi'+x\xi)} \cdot e^{\pi i ( x'' \xi' -x'\xi'' )}e^{-\pi i x''\xi''}
		\\ &=:F(w-z') \cdot G_z (z') \cdot e^{\pi i ( x'' \xi' -x'\xi'' )}e^{-\pi i x''\xi''}.
	\end{align*}
	So,
	\begin{align*}
	\int_{E^c}K(w,z')K(z',z) dz' &= e^{-\pi i x''\xi''} \int_{\R^{2d}} \rchi_{E^c}(z') G_z (z') F(w-z') e^{\pi i ( x'' \xi' -x'\xi'' )}dz'
	\\ &= e^{-\pi i x''\xi''} (\rchi_{E^c}G_z)\natural F(w).
	\end{align*}
	Invoking \eqref{eq_tc}, we obtain
	\begin{multline}\label{eq_i1}
		\int_{\R^{2d}}\Big|\int_{E^c}K(w,z')K(z',z) dz'\Big|^2 dw =\|(\rchi_{E^c}G_z)\natural F\|_2^2 
		\le \|F\|_2^2 \|\rchi_{E^c} G_z\|_2^2 
		\\=\|V_g g\|_2^2 \int_{E^c}|V_g g (z'-z)|^2 dz'
		= \int_{E^c}|V_g g (z'-z)|^2 dz', 
	\end{multline}
	and it follows that $I_1\le I_3$. Regarding $I_2$, we observe that, by the reproducing formula \eqref{eq_rep},
	\[\int_{E} K (w,z')K(z',z) \, dz' = K(w,z)-\int_{E^c} K (w,z')K(z',z) \, dz'.\]
	By the triangle inequality 
	and applying again \eqref{eq_i1},
	\begin{align*}
	\Big(\int_{E^c} \Big|\int_{E} K (w,z')K(z',z) \, dz'\Big|^{2}\, dw\Big)^{1/2}\! \le \, & \Big(\int_{E^c} |K(w,z)|^{2}\, dw\Big)^{1/2}
	\\ &+\Big(\int_{\R^{2d}}\Big|\int_{E^c} K (w,z')K(z',z) \, dz'\Big|^{2} dw\Big)^{1/2}
	\\ \le \, & 2 \Big( \int_{E^c} | V_g g (z'-z)|^2 \, dz'\Big)^{1/2} .
	\end{align*}
	This directly implies that $I_2\le 2^p  I_3 \le 4  I_3$.
	
	\noindent {\bf Step 3}. By \eqref{eq_b1} and \eqref{eq_b2},
	\begin{align} \label{eq2}
	\|H\|_p^p\lesssim  \max\{I_3(E): E=\Omega, \Omega^C\}.
	\end{align}
	We now bound $I_3$ for either $E=\Omega, \Omega^C$. Recall the parameter $\alpha>0$ and define $h:\R^{2d}\to\R$ as \[h(z)=(1+d(z,\partial \Omega)\eta^{-1})^{2d-1}(1+d(z,\partial \Omega))^{1+\alpha}. \]
	By H\"{o}lder's inequality,
	\begin{align}\label{i3}
	I_3 &=\int_{E} \frac{h(z)^{1-p/2}}{h(z)^{1-p/2}} \Big(\int_{E^c} | V_g g (z'-z)|^2 \, dz'\Big)^{p/2}  \, dz \notag
	\\ &\le \Big( \int_{E} \frac{1}{h(z)} \, dz\Big)^{1-p/2} \Big(\int_{E}\int_{E^c} h(z)^{(2-p)/p}| V_g g(z'-z)|^2 \, dz'dz\Big)^{p/2} .
	\end{align}
	We study the two integrals in \eqref{i3} separately. For the first one we use the coarea formula and invoke Proposition \ref{prop_ca} to obtain
	\begin{align}\label{coa}
	\begin{aligned}
	\int_{E} \frac{1}{h(z)} \, dz \le & \int_{\R^{2d}} \frac{1}{h(z)} \, dz= \int_{\R^{2d}} \frac{|\nabla d(z,\partial \Omega)|}{h(z)} \, dz
	= \int_0^\infty \frac{\mathcal{H}^{2d-1}(\{w :\ d(w,\partial \Omega)=r\})}{(1+r\eta^{-1})^{2d-1}(1+r)^{1+\alpha}} \, dr
	\\
	&\lesssim \frac{\mathcal{H}^{2d-1}(\partial \Omega)}{\cu} \int_0^\infty \frac{1}{(1+r)^{1+\alpha}} \, dr = \frac{\mathcal{H}^{2d-1}(\partial \Omega)}{\cu  \cdot \alpha}.
	\end{aligned}
	\end{align}
	Next we bound the second integral in \eqref{i3}. Since $E$ is either $\Omega$ or $\Omega^C$, for $z \in E, z' \in E^C$,
	\begin{align*}
	h(z)&=(1+d(z,\partial \Omega)\eta^{-1})^{2d-1}(1+d(z,\partial \Omega))^{1+\alpha}
	\le (1+|z'-z|\eta^{-1})^{2d-1}(1+|z'-z|)^{1+\alpha}.
	\end{align*}
	So we get
	\begin{multline*}
	\int_{E}\int_{E^c} h(z)^{(2-p)/p}| V_g g (z'-z)|^2 \, dz'dz 
	\\ \le \int_{E}\int_{E^c} (1+|z'-z|\eta^{-1})^{(2d-1)(2-p)/p}(1+|z'-z|)^{(1+\alpha)(2-p)/p}| V_g g (z'-z)|^2 \, dz'dz. \notag
	\end{multline*}
	Setting $\varphi(z)=(1+|z|\eta^{-1})^{(2d-1)(2-p)/p}(1+|z|)^{(1+\alpha)(2-p)/p}| V_g g (z)|^2$,
	Lemma \ref{lemma_var} gives
	\begin{align}\label{var}
	\begin{aligned}
	\int_{E}\int_{E^c} h(z)^{(2-p)/p}| V_g g(z'-z)|^2 \, dz'dz 
	&\le \int_{E^c} \int_{E} \varphi (z'-z) \, dz dz'.
	\\ &\le \big\| \left(\smallint \varphi\right)\rchi_{E^C}-\rchi_{E^C}\ast \varphi \big\|_1
	\\
	&\le \mathcal{H}^{2d-1}(\partial E) \int_{\R^{2d}} |z||\varphi(z)| \, dz.	
	\end{aligned}
	\end{align}
	Combining \eqref{i3}, \eqref{coa} and \eqref{var} we obtain
	\begin{multline*}
	I_3 \lesssim \frac{\mathcal{H}^{2d-1}(\partial \Omega)}{(\cu \cdot \alpha)^{1-p/2}} \|z\varphi\|_1^{p/2}
	\\ \le \frac{\mathcal{H}^{2d-1}(\partial \Omega)}{(\cu \cdot  \alpha)^{1-p/2}} \big\lVert (1+|z|\eta^{-1})^{(2d-1)(2-p)/p}(1+|z|)^{(1+\alpha)(2-p)/p+1}|V_g g|^2\big\rVert_1^{p/2},
	\end{multline*}
	which, together with \eqref{eq2}, yields \eqref{eq_p21}.
\end{proof}

\section{Eigenvalue estimates}\label{sec_eig}
\subsection{Spectral deviation estimates}\label{sec_main}
The next lemma allows us to describe the eigenvalue counting function in terms of Schatten norms of Hankel operators.
\begin{lemma} \label{lemaautoval}
	Let $\lVert{g}\rVert_{L^2(\mathbb{R}^d)}=1$, $\Omega\subseteq \R^{2d}$ be compact, and let $H$ be the corresponding Gabor-Hankel operator. Then, for every $\delta\in(0,1)$ and every $0<p \leq 2$,
	\begin{align*}
	\big|\#\{\lambda \in \sigma(\lo_\Omega): \lambda > \delta \}-|\Omega|\big|
	\leq (\delta(1-\delta))^{-p/2}\cdot \|H\|_p^p,
	\end{align*}
\end{lemma}
\begin{proof}
	Define $G:[0,1]\to\R$ by
	\[G(t)=\begin{cases} -t & \text{ if } 0\le t \le \delta \\ 1-t & \text{ if } \delta < t \le 1 \end{cases}.
	\]
	Since $|G(t)|\le1$ for $t\in[0,1]$, a straightforward concavity argument shows that
	\[|G(t)|\le \left(\frac{t-t^2}{\delta-\delta^2}\right)^{p/2},   \qquad t\in[0,1].\] 
	Recall that the Gabor-Toeplitz operator $T$ is related to $\lo_{ \Omega}$ by \eqref{eq_conj}, and to $H$ by \eqref{eq_ht}. Since $0\le T \le 1$, we conclude
	\begin{align*}
\big|\#\{\lambda \in \sigma(\lo_\Omega): \lambda > \delta \}-|\Omega|\big|
	&= |\tr(G(T))|\le \tr(|G|(T)) \notag
	\\ &\le (\delta-\delta^2)^{-p/2} \cdot \tr\left[(T-T^2)^{p/2}\right]
	\\ &=  (\delta(1-\delta))^{-p/2} \cdot \|H\|_p^{p} . \qedhere
	\end{align*}
\end{proof}

We can now prove our first main result.
\begin{proof}[Proof of Theorem \ref{th_intro_1}]
Let $0<p\le 1/2$ and $0<\alpha\le 1$. By Lemma \ref{lemaautoval},
\begin{align} \label{tau}
\begin{aligned}
\left|\#\{\lambda \in \sigma(\lo_\Omega): \ \lambda>\delta\}-|\Omega|\right| 
&\le (\delta(1-\delta))^{-p/2} \|H\|_p^p 
\\
&\leq 2^{p/2} \max\{\delta^{-1},(1-\delta)^{-1}\}^{p/2} \|H\|_p^p 
\\ &\leq 2 \max\{\delta^{-1},(1-\delta)^{-1}\}^{p/2} \|H\|_p^p.
\end{aligned}
\end{align}
We use Proposition \ref{prop1} to bound $\|H\|_p^p$. First note that
\begin{align*}
	(1+|z|\eta^{-1})^{(2d-1)(2-p)/p}&\le (1+|z|\eta^{-1})^{(2d-1)2/p} \le 2^{(2d-1)2/p} \big(1+(|z|\eta^{-1})^{(2d-1)2/p}\big)
	\\ &\le 2^{(2d-1)2/p} \big(1+(\eta^{-1}(1+|z|))^{(2d-1)2/p}\big),
\end{align*}
and therefore,
\begin{align}\label{eq_mul}
	\begin{multlined}[\linewidth]
	 \left\|(1+|z|\eta^{-1})^{(2d-1)(2-p)/2p}(1+|z|)^{(1+\alpha)(2-p)/2p+1/2}V_g g\right\|_2^{p} 
\\ 
\hfill \begin{aligned}
=		&\left\|(1+|z|\eta^{-1})^{(2d-1)(2-p)/2p}(1+|z|)^{(1+\alpha)/p-\alpha/2}V_g g\right\|_2^{p} 
\\ \lesssim &\left(\int_{\R^{2d}}(1+|z|)^{(1+\alpha)2/p-\alpha}|V_g g(z)|^2 dz +
	\frac{1}{\eta^{(2d-1)2/p}} \int_{\R^{2d}}(1+|z|)^{(2d+\alpha)2/p-\alpha} |V_g g(z)|^2 dz\right)^{p/2} . 
	\end{aligned}	
	\end{multlined}		
\end{align}
Regarding the first integral we have
\begin{align}\label{depe2}
\begin{aligned}
	\int_{\R^{2d}}(1+|z|)^{(1+\alpha)2/p-\alpha}|V_g g(z)|^2 \, dz  &= \int_{\R^{2d}} \frac{\left((1+|z|)^{(1+\alpha)/p +d}| V_g g (z)|\right)^2}{(1+|z|)^{2d+\alpha}} \, dz
	\\ &\le \left\|(1+|z|)^{-(2d+\alpha)}\right\|_1 \left\|(1+|z|)^{(1+\alpha)/p +d} V_g g \right\|_\infty^2  
	\\ & \lesssim \frac{1}{\alpha }  \left\|(1+|z|)^{(1+\alpha)/p +d} V_g g \right\|_\infty^2 .
	\end{aligned}
\end{align}

By \eqref{eq_gs}, we get
\begin{align} \label{eq_ctes2}
\begin{aligned}
	\left\|(1+|z|)^{(1+\alpha)/p +d} V_g g \right\|_\infty^2
	&\le C_g^2 A^{2\ceil*{(1+\alpha)/p+d}} (\ceil*{(1+\alpha)/p+d})!^{2\beta }
	\\ &\le C_g^2 A^{2(1+\alpha)/p +2d+2} \left(\frac{1+\alpha+p(d+1)}{p}\right)^{2\beta ((1+\alpha)/p+d+1)}
	\\ & \le C_g^{2} A^{4/p+2d+2} \left(d +3\right)^{2\beta (2/p +d+1)}  p^{-2\beta(d+1)}   p^{-2\beta(1+\alpha)/p} .
	\end{aligned}
\end{align}

Similarly, for the second integral in the right-hand side of \eqref{eq_mul},
\begin{align}\label{depe}
\begin{aligned}
\int_{\R^{2d}}(1+|z|)^{(2d+\alpha)2/p-\alpha} |V_g g(z)|^2 \, dz &= \int_{\R^{2d}} \frac{\left((1+|z|)^{(2d+\alpha)/p +d}| V_g g (z)|\right)^2}{(1+|z|)^{2d+\alpha}} \, dz
	\\ &\le \left\|(1+|z|)^{-(2d+\alpha)}\right\|_1 \left\|(1+|z|)^{(2d+\alpha)/p +d} V_g g \right\|_\infty^2 
\\ & \lesssim \frac{1}{\alpha }  \left\|(1+|z|)^{(2d+\alpha)/p +d} V_g g \right\|_\infty^2 .
\end{aligned}
\end{align}
Again by \eqref{eq_gs}, we get
\begin{align} \label{eq_ctes}
\left\|(1+|z|)^{(2d+\alpha)/p +d} V_g g \right\|_\infty^2
&\le C_g^2 A^{2\ceil*{(2d+\alpha)/p+d}} (\ceil*{(2d+\alpha)/p+d})!^{2\beta }
\\ &\le C_g^2 A^{2(2d+\alpha)/p +2d+2} \left(\frac{2d+\alpha+p(d+1)}{p}\right)^{2\beta ((2d+\alpha)/p+d+1)} \notag
\\ & \le C_g^{2} A^{2(2d+1)/p+2d+2} \left(3d +2\right)^{2\beta ((2d+1)/p +d+1)}  p^{-2\beta(d+1)}   p^{-2\beta(2d+\alpha)/p} . \notag
\end{align}
Joining \eqref{eq_mul}, \eqref{depe2}, \eqref{eq_ctes2}, \eqref{depe} and \eqref{eq_ctes} we obtain
\begin{align}\label{eq_mul2}
\begin{multlined}[\linewidth]
	\big\|(1+|z|\eta^{-1})^{(2d-1)(2-p)/2p} (1+|z|)^{(1+\alpha)(2-p)/2p+1/2}V_g g\big\|_2^{p} 
\\	
\begin{aligned}
 &\lesssim \alpha^{-p/2} C_g^{p} A^{2d+1+(d+1)p} \left(3d +2\right)^{\beta (2d+1 +(d+1)p)} p^{-p\beta(d+1)} p^{-\beta(1+\alpha)} \big(1+(\eta p^\beta)^{-(2d-1)2/p}\big)^{p/2} 
\\	& \lesssim \alpha^{-p/2} C_g^{1/2} A^{3d+2} \left(3d +2\right)^{\beta (3d+2)} e^{\beta(d+1)} p^{-\beta(1+\alpha)} \big(1+(\eta p^\beta)^{-(2d-1)}\big) 
\\	& \lesssim \alpha^{-p/2} C_g^{1/2} A^{3d+2} C_d^\beta p^{-\beta(1+\alpha)} \big(1+(\eta p^\beta)^{-(2d-1)}\big),
\end{aligned}
\end{multlined}
\end{align}
where we used the elementary bound $p^{-p} \leq e$.
($C_d$ always denotes a constant that depends only on $d$; its particular value may change from line to line.)
Combining this with Proposition~\ref{prop1} and \eqref{eq_kappa_2}, we get
\begin{align}\label{eq_H}
\begin{aligned}
\|H\|_p^p  
&\lesssim \frac{ \big\|(1+|z|\eta^{-1})^{(2d-1)(2-p)/2p} (1+|z|)^{(1+\alpha)(2-p)/2p+1/2}V_g g\big\|_2^{p}}{(\cu \cdot \alpha)^{1-p/2}} \cdot \mathcal{H}^{2d-1}(\partial \Omega) 
\\ &\lesssim \frac{C_g^{1/2} A^{3d+2} C_d^\beta }{\cu  } \cdot
\frac{1+(\eta p^\beta)^{-(2d-1)}}{ \alpha p^{\beta(1+\alpha)} } \cdot \mathcal{H}^{2d-1}(\partial \Omega).
\end{aligned}
\end{align}

We now choose $\alpha=1/(\log(1/p+1))$ and assume for the moment that this choice is indeed compatible with the assumption $\alpha\le 1$, to obtain
\begin{align*}
\|H\|_p^p     &\lesssim  \frac{C_g^{1/2} A^{3d+2} C_d^\beta e^\beta}{\cu } \cdot
\frac{\big(1+(\eta p^\beta)^{-(2d-1)}\big)\log(1/p+1)}{ p^{\beta}} \cdot \mathcal{H}^{2d-1}(\partial \Omega). \notag
\end{align*}
This together with \eqref{tau} gives
\[\big|\#\{\lambda \in \sigma(\lo_\Omega): \ \lambda>\delta\}-|\Omega|\big| 
\lesssim  \frac{C_g^{1/2} A^{3d+2} C_d^\beta}{\cu }\cdot
\frac{\tau^{p/2} \big(1+(\eta p^\beta)^{-(2d-1)}\big) \log(1/p+1)}{ p^{\beta}} \cdot \mathcal{H}^{2d-1}(\partial \Omega).\]
Taking $p=1/\log \tau$ yields \eqref{eq_t1}, provided that this choice indeed leads to $p \leq 1/2$ and $\alpha\le 1$.

To complete the proof, we first observe that, if $\tau \ge e^2$, then, indeed,
\begin{align*}
p := \frac{1}{\log \tau} \le \frac{1}{2}, 
\quad \text{and} \quad
\alpha :=\frac{1}{\log \big((1/p)+1\big)} \leq \frac{1}{\log (2+1)} \le  1,
\end{align*}
are valid choices for $p$ and $\alpha$.
On the other hand, if $\tau \le e^2$ we choose $p=1/2$ and $\alpha=1$ in \eqref{eq_H} to get
\begin{align}\label{eq_mmm}
\begin{aligned}
\|H\|_p^p     &\lesssim  \frac{C_g^{1/2} A^{3d+2} C_d^\beta}{\cu }  \cdot  2^{2\beta}\big(1+2^{\beta(2d-1)}\eta^{-(2d-1)}\big) \cdot \mathcal{H}^{2d-1}(\partial \Omega) 
\\ &\lesssim  \frac{C_g^{1/2} A^{3d+2} C_d^\beta}{\cu }  \cdot  \big(1+\eta^{-(2d-1)}\big) \cdot \mathcal{H}^{2d-1}(\partial \Omega).
\end{aligned}
\end{align}
Note that $\tau \geq 2$, as either $\delta \leq 1/2$ or $1-\delta \leq 1/2$. Combining this observation, the assumption $\tau \leq e^2$, \eqref{tau} and \eqref{eq_mmm}, we conclude
\begin{align*}
\big|\#&\{\lambda \in \sigma(\lo_\Omega): \ \lambda>\delta\}-|\Omega|\big| 
\lesssim \frac{C_g^{1/2} A^{3d+2} C_d^\beta}{\cu }  \cdot \big(1+\eta^{-(2d-1)}\big) \cdot \tau^{1/4} \cdot \mathcal{H}^{2d-1}(\partial \Omega)
\\ &\le  \frac{C_g^{1/2} A^{3d+2} C_d^\beta}{\cu } \cdot \Big(1+\Big(\frac{\log \tau}{\log 2 }\Big)^{\beta(2d-1)}\eta^{-(2d-1)}\Big) \cdot \frac{\sqrt{e} (\log \tau)^{\beta } 	\log\big[\log(\tau)+1\big]}{(\log 2)^{\beta} 	\log\big[\log(2)+1\big]} \cdot \mathcal{H}^{2d-1}(\partial \Omega) 
\\ &\le  \frac{C_g^{1/2} A^{3d+2} C_d^\beta}{\cu } \cdot \Big(1+\Big(\frac{(\log \tau)^\beta}{\eta }\Big)^{(2d-1)}\Big)  (\log \tau)^{\beta } 	\log\big[\log(\tau)+1\big] \cdot \mathcal{H}^{2d-1}(\partial \Omega) .
\end{align*}
Hence, we have proved \eqref{eq_t1} for all possible values of $\tau$.
\end{proof}
\begin{rem}
The constant $C'_g$ in \eqref{eq_t1} can be taken as
\[C'_g = C_d \cdot C_g^{1/2} \cdot A^{3d+2} \cdot e^{\beta(d+2)} \cdot (3d+2)^{\beta(3d+2)} \cdot 2^{\beta (2d+1)} \cdot (\log 2)^{-2 d \beta} .\]
In fact, assuming that $A\ge 1$, a close inspection of \eqref{eq_ctes2}$-$\eqref{eq_mul2} 
shows that for sufficiently large $\tau$, \eqref{eq_t1} holds with $C'_{g}=C_d  \cdot A^{2d} \cdot e^{\beta} \cdot (2d)^{2d\beta}$.
\end{rem}

Finally, we prove our second main result.

\begin{proof}[Proof of Theorem \ref{th_2}]
Fix $\alpha >0$ and set
\[p=2\frac{2d+\alpha}{2d+\alpha+s-1}, \qquad \text{so that} \qquad s=1+\frac{(2d+\alpha)(2-p)}{p}.\]
Note that $p\le 2$ since $s-1\ge 0$.
Applying Proposition \ref{prop1} we obtain
\begin{align*}
    \|H\|_p^p  
    &\lesssim \mathcal{H}^{2d-1}(\partial \Omega) \cdot \frac{ \left\|(1+|z|)^{(2d+\alpha)(2-p)/2p+1/2}V_g g\right\|_2^{p}}{(\cu \cdot \eta^{2d-1} \cdot \alpha)^{1-p/2}}
   = \mathcal{H}^{2d-1}(\partial \Omega) \cdot \frac{C_g^{p/2}}{(\cu \cdot \eta^{2d-1} \cdot \alpha)^{1-p/2}}.
\end{align*}
As in the proof of Theorem \ref{th_intro_1}, by Lemma
\ref{lemaautoval}, we derive \eqref{tau}. In combination with the previous estimate this leads to the bound
\begin{align}\label{eq_choclo}
   \big|\#\{\lambda \in \sigma(\lo_\Omega): \lambda>\delta\}-|\Omega|\big|
    \lesssim \mathcal{H}^{2d-1}(\partial \Omega)\cdot    \frac{(C_g\tau)^{p/2} }{ (\cu \cdot \eta^{2d-1} \cdot \alpha)^{1-p/2}}.
\end{align}
We now choose $\alpha$ so that the above expression is small.
Note that $C_g\tau\ge 2$ since $\|V_gg\|_2=1$ and $\tau=\max\{\delta^{-1},(1-\delta)^{-1}\}\ge 2$. Therefore, we can choose 
\[\alpha=\frac{\log 2}{\log (C_g\tau)}\le 1.\]
Let us now estimate \eqref{eq_choclo} for this choice of $\alpha$.
First note that
\begin{align*}
\frac{p}{2}&=\frac{2d+\alpha}{2d+\alpha + s-1 }
=\frac{2d}{2d+s-1 + \alpha}+\frac{\alpha}{2d+s-1 +\alpha}
\le \frac{2d}{2d+s-1} + \frac{\log 2}{\log(C_g\tau)},
\end{align*}
and, consequently,
\begin{align}
\label{eq_taup2}
(C_g \tau)^{p/2}\leq 2 (C_g \tau)^{2d/{(2d+s-1)}}.
\end{align}
In addition, we have,
\begin{align*}
1-\frac{p}{2}=\frac{s-1}{2d+s-1+\alpha}\le \frac{s-1}{2d+s-1} \le 1.
\end{align*}
Hence, using \eqref{eq_kappa_2}, we obtain:
\begin{align}
\label{eq_cte}
\begin{aligned}
\bullet \ &(\cu \cdot \eta^{2d-1})^{1-p/2}\gtrsim  (\cu \cdot \eta^{2d-1})^{(s-1)/(2d+s-1)} ; 
\\ \bullet \ &\left(\frac{\log (C_g \tau)}{\log 2}\right)^{1-p/2}\le \left(\frac{\log (C_g \tau)}{\log 2}\right)^{(s-1)/(2d+s-1)}
\le \frac{(\log (C_g \tau))^{(s-1)/(2d+s-1)}}{\log 2}.
\end{aligned}
\end{align}
Finally, \eqref{eq_tpoli} follows by applying the estimates \eqref{eq_taup2} and \eqref{eq_cte} to  \eqref{eq_choclo}.
\end{proof}

\subsection{Asymptotics for eigenvalues}\label{sec_eigasy}

We now apply Theorems \ref{th_intro_1} and \ref{th_2} to derive asymptotic estimates for the individual eigenvalues of concentration operators with compact time-frequency domains. This is a standard procedure, which we present in detail for completeness.

We order the eigenvalues
of $\lo_{ \Omega}$ in decreasing order. More precisely, we define
\[\lambda_k=\inf\{\|\lo_\Omega-S\| : \ S\in \mathcal{L}(L^2(\R^d)),\,\dim(\im S)<k\},
\qquad  k \geq 1.
\]
As $\lo_{ \Omega}$ is compact, $\{ \lambda_k: k \geq 1 \} \setminus \{0\} = \sigma(\lo_{ \Omega}) \setminus \{0\}$ as sets with multiplicities --- see, e.g., \cite[Lemma 4.3]{dj}.

As a first step towards eigenvalue asymptotics, we estimate the index $k$ where the eigenvalue $\lambda_k$ crosses the threshold $1/2$. Theorems ~\ref{th_intro_1} and~\ref{th_2} suggest that this occurs near the index $A_\Omega = \lceil |\Omega| \rceil$ -- that is, the smallest integer $\geq \abs{\Omega}$.

\begin{lemma}\label{lemplunge}
Let $\Omega\subseteq \R^{2d}$ be a compact set, $g\in L^2(\R^d)$ with $\|g\|_2=1$, and denote
\begin{align}\label{eq_kg}
\newkg = 2 \int_{\mathbb{R}^{2d}} |z| |V_gg(z)|^2 \, dz.
\end{align}
The following statements hold:
\begin{enumerate}[label=(\roman*),ref=(\roman*)]
	\item\label{lemplungei} $\lambda_k\le 1/2$ for every $k\ge A_\Omega + \newkg \cdot \mathcal{H}^{2d-1}(\partial \Omega),$
	\item\label{lemplungeii} $\lambda_k\ge 1/2$ for every $1\le k \le A_\Omega - \newkg \cdot \mathcal{H}^{2d-1}(\partial \Omega).$
\end{enumerate}

\end{lemma}
\begin{proof}
Without loss of generality we can assume that $0<\newkg \cdot \mathcal{H}^{2d-1}(\partial \Omega)<\infty$. Indeed, if $\newkg \cdot \mathcal{H}^{2d-1}(\partial \Omega)=\infty$, then the statements are vacuous. On the other hand, $\newkg>0$ as $\|V_g g\|_2=\|g\|_2=1$, while, by the isoperimetric inequality, $\mathcal{H}^{2d-1}(\partial \Omega)=0$ can only occur if $|\Omega|=0$ and $A_\Omega=0$ (see, e.g., \cite[Section 5.2.6, Theorem 2]{evga92}), in which case the conclusions hold trivially. Thus, we assume that $0<\newkg \cdot \mathcal{H}^{2d-1}(\partial \Omega)<\infty$.

By~\eqref{eq_traces} and Lemma \ref{lemma_var},
\begin{align}\label{eq_traz}
\begin{aligned}
0\le \tr(\lo_\Omega)-\tr(\lo_\Omega^2) &=\int_\Omega 1 - \left(1_\Omega*|V_g g|^2\right)(z) \, dz
\\ &\le \left\|1_\Omega-1_\Omega *|V_g g|^2\right\|_1\le \frac12 \cdot \newkg \cdot \mathcal{H}^{2d-1}(\partial \Omega).
\end{aligned}
\end{align}
Second, we proceed as in the proof of \cite[Theorem~1.5]{AbPeRo2}. We have
\begin{align*}
\tr(\lo_\Omega)-\tr(\lo_\Omega^2)&= \sum_{n=1}^\infty \lambda_n(1-\lambda_n)
=\sum_{n=1}^{A_\Omega} \lambda_n(1-\lambda_n) +\sum_{n=A_\Omega+1}^\infty \lambda_n(1-\lambda_n)
\\ &\ge \lambda_{A_\Omega} \sum_{n=1}^{A_\Omega} (1-\lambda_n) +(1-\lambda_{A_\Omega})\sum_{n=A_\Omega+1}^\infty \lambda_n
\\ &=\lambda_{A_\Omega} A_\Omega  - \lambda_{A_\Omega} \sum_{n=1}^{A_\Omega} \lambda_n + (1-\lambda_{A_\Omega}) \Big(|\Omega|-\sum_{n=1}^{A_\Omega} \lambda_n\Big)
\\ &= \lambda_{A_\Omega} A_\Omega  -  \sum_{n=1}^{A_\Omega} \lambda_n + (1-\lambda_{A_\Omega}) |\Omega|
= |\Omega|-  \sum_{n=1}^{A_\Omega} \lambda_n + \lambda_{A_\Omega}(A_\Omega -  |\Omega|).
\end{align*}
Combining this with \eqref{eq_traz} we obtain
\[|\Omega|-\sum_{n=1}^{A_\Omega}\lambda_n+\lambda_{A_\Omega}(A_\Omega-|\Omega |)\le \mathrm{trace}(\lo_\Omega)-\mathrm{trace}(\lo_\Omega^2)\le \frac12 \cdot \newkg \cdot \mathcal{H}^{2d-1}(\partial \Omega). \]
As a consequence,
\begin{align}
\sum_{n=A_\Omega+1}^\infty \lambda_n &=|\Omega|-\sum_{n=1}^{A_\Omega}\lambda_n \le  \frac12 \cdot \newkg \cdot \mathcal{H}^{2d-1}(\partial \Omega), \qquad\text{and} \label{eqaom1}
\\ \sum_{n=1}^{A_\Omega-1}(1-\lambda_n) &=|\Omega|-\sum_{n=1}^{A_\Omega}\lambda_n +\lambda_{A_\Omega}(A_\Omega-|\Omega |) -(1+|\Omega|-A_\Omega)(1-\lambda_{A_\Omega}) \label{eqaom2}
\\ &\le  \frac12 \cdot \newkg \cdot \mathcal{H}^{2d-1}(\partial \Omega).\notag
\end{align}
To prove \ref{lemplungei}, let $k \geq A_\Omega + \newkg \cdot \mathcal{H}^{2d-1}(\partial \Omega)$.
We can write
\begin{align*}
k = A_\Omega + j,
\end{align*}
with $j \geq  \newkg \mathcal{H}^{2d-1}(\partial \Omega)$, and use \eqref{eqaom1} to estimate
\begin{align*}
\newkg \cdot \mathcal{H}^{2d-1}(\partial \Omega)
\cdot \lambda_{k} &\leq j \cdot \lambda_{A_\Omega+j}
\leq \sum_{n=A_\Omega+1}^\infty \lambda_n \leq \frac12 \cdot \newkg \cdot \mathcal{H}^{2d-1}(\partial \Omega).
\end{align*}
Since $0<\newkg \cdot \mathcal{H}^{2d-1}(\partial \Omega)<\infty$, it follows that $\lambda_k \leq 1/2$, as claimed.

We proceed similarly to prove \ref{lemplungeii}. Let $1\le k \le A_\Omega - \newkg \cdot \mathcal{H}^{2d-1}(\partial \Omega)$. 
We write
\begin{align*}
k=A_\Omega - j,
\end{align*}
with $j \geq  \newkg \mathcal{H}^{2d-1}(\partial \Omega)$, and use \eqref{eqaom2} to estimate
\begin{align*}
\newkg \cdot \mathcal{H}^{2d-1}(\partial \Omega) \cdot
(1-\lambda_k) \leq j (1-\lambda_{A_\Omega-j})
\le \sum_{n=1}^{A_\Omega-1} (1-\lambda_n)
\leq \frac12 \cdot \newkg \cdot \mathcal{H}^{2d-1}(\partial \Omega).
\end{align*}
Again since $0<\newkg \cdot \mathcal{H}^{2d-1}(\partial \Omega)<\infty$, it follows that $\lambda_k \geq 1/2$ as claimed.
\end{proof}

We can now prove eigenvalue asymptotics in the context to Theorem \ref{th_intro_1}.

\begin{proof}[Proof of Corollary \ref{coro_1}]
A direct computation shows that
\begin{align}\label{eq_aom}
\newkg \le C_g',
\end{align}
where $\newkg$ is the magnitude defined in \eqref{eq_kg}, and the dimensional constant $C_d$ is large enough.

To prove \ref{coro_1i}, let $k=A_\Omega+\gamma h\in\N$ with $h\ge 1$. By \eqref{eq_aom},
\begin{align*}
A_\Omega + \newkg \cdot \mathcal{H}^{2d-1}(\partial \Omega)\le A_\Omega+ \gamma \le k.
\end{align*}
Thus, by Lemma~\ref{lemplunge}, $\lambda_k\le 1/2$. We may assume that $\lambda_k>0$, since otherwise \eqref{eqeig1} holds trivially.
Applying Theorem~\ref{th_intro_1} for $0<\delta\nearrow \lambda_k\le 1/2$ and $\eta\le 1$, we obtain
\begin{align*}
k-A_\Omega &\le \lim_{\delta\nearrow \lambda_k} \big| \#\sett{\lambda \in \sigma(\lo_{ \Omega}): \lambda > \delta} - \abs{\Omega} \big|
\le \gamma  \lim_{\delta\nearrow \lambda_k} (\log (1/\delta))^{2d\beta}	\log\big[\log(1/\delta)+1\big]
\\ & = \gamma  (\log (1/\lambda_k))^{2d\beta}	\log\big[\log(1/\lambda_k)+1\big] \le \frac{\gamma}{\varepsilon}  (\log (1/\lambda_k))^{2d\beta +\varepsilon},
\end{align*}
for every $0<\varepsilon\le 1$, where we used that $\log(x+1)\le x^\varepsilon/\varepsilon$ for $x\ge0$. Rearranging the last expression we arrive at
\begin{align}\label{eqh}
(\varepsilon h)^{1/(2d\beta+\varepsilon)}\le \log(1/\lambda_k).
\end{align}
(As in the proof of Lemma \ref{lemplunge}, we assume without loss of generality that $\mathcal{H}^{2d-1}(\partial \Omega)>0$, since, otherwise, the conclusions hold trivially.) Choosing $\varepsilon=1/(1+\log h)\le 1$,
\begin{align*}
\Big(\frac{h}{1+\log h}\Big)^{1/(2d\beta)}&=(\varepsilon h)^{1/(2d\beta+\varepsilon)}(\varepsilon h)^{1/(2d\beta)-1/(2d\beta+\varepsilon)}
=(\varepsilon h)^{1/(2d\beta+\varepsilon)}(\varepsilon h)^{\varepsilon/(2d\beta)(2d\beta+\varepsilon)}
\\ & \le (\varepsilon h)^{1/(2d\beta+\varepsilon)}h^{1/(1+\log h)(2d\beta)^2}
\le (\varepsilon h)^{1/(2d\beta+\varepsilon)}e^{1/(2d\beta)}.
\end{align*}
Combining this with \eqref{eqh} yields
\begin{align*}
\left(\frac{h}{e(1+\log h)}\right)^{1/(2d\beta)}&\le \log(1/\lambda_k),
\\ e^{\left(\frac{h}{e(1+\log h)}\right)^{1/(2d\beta)}} &\le  \frac{1}{\lambda_k},
\end{align*}
and \eqref{eqeig1} follows. 

Towards \ref{coro_1ii}, let $k=A_\Omega-\gamma h\in\N$ with $h\ge 1$.
By \eqref{eq_aom},
\[k= A_\Omega - \gamma h \le A_\Omega- \gamma \le A_\Omega -\newkg \cdot \mathcal{H}^{2d-1}(\partial \Omega).\]
From Lemma~\ref{lemplunge} it follows that $\lambda_k\ge 1/2$. 
Applying Theorem~\ref{th_intro_1}, we obtain
\begin{align*}
A_\Omega-k &\le |\Omega| - (k-1) \le \big|  \abs{\Omega} - \#\sett{\lambda \in \sigma(\lo_{ \Omega}): \lambda > \lambda_k} \big|
\\ &\le \gamma    (\log (1/(1-\lambda_k)))^{2d\beta}	\log\big[\log(1/(1-\lambda_k))+1\big] 
\\ &\le \frac{\gamma}{\varepsilon}  (\log (1/(1-\lambda_k)))^{2d\beta +\varepsilon}, \qquad 0<\varepsilon\le 1.
\end{align*}
Proceeding exactly as before we arrive at
\begin{align*}
\Big(\frac{h}{e(1+\log h)}\Big)^{1/(2d\beta)}&\le \log(1/(1-\lambda_k))
\\ e^{\left(\frac{h}{e(1+\log h)}\right)^{1/(2d\beta)}} &\le  \frac{1}{1-\lambda_k},
\end{align*}
which proves \eqref{eqeig2}.
\end{proof}

We now prove an analogous result concerning windows with polynomial time-frequency decay.

\begin{proof}[Proof of Corollary \ref{coro_2}]
Recall that the constant $C_g'$ in \eqref{eq_gamma} is given by $C_g'=C_d \cdot C_g^{\frac{2d}{2d+s-1}}$ (see Theorem~\ref{th_2}). We assume, as we may, that $C_d$  satisfies $C_d \geq 2$.
Second, since $s\ge 1$ we have that $2d\leq 2d+s-1\le 2ds$.
As $||V_gg||_2=1$, Jensen's inequality gives
the following bound on the quantity in \eqref{eq_kg}:
\begin{align*}
 \newkg = 2
\int_{\mathbb{R}^{2d}} |z| |V_g g(z)|^2 \,dz & \le 2 \int_{\mathbb{R}^{2d}} (1+|z|) |V_g g(z)|^2 \,dz
\\ &\le 2 \left(\int_{\mathbb{R}^{2d}} (1+|z|)^{\frac{2d+s-1}{2d}} |V_g g(z)|^2 \,dz\right)^{\frac{2d}{2d+s-1}}
\\ &\le 2 \left(\int_{\mathbb{R}^{2d}} (1+|z|)^{s} |V_g g(z)|^2 \,dz\right)^{\frac{2d}{2d+s-1}}
\\ &= 2 \cdot C_g^{\frac{2d}{2d+s-1}}\le C_g'.
\end{align*}
So we conclude:
\begin{align}\label{eq_aom2}
\newkg \cdot \mathcal{H}^{2d-1}(\partial \Omega)\le C_g' \cdot \mathcal{H}^{2d-1}(\partial \Omega) \le \gamma \le \gamma h,
\qquad \mbox{for } h \geq 1.
\end{align}

To prove \ref{coro_2i}, we let $k=A_\Omega+\gamma h\in\N$ with $h\ge 1$ and use \eqref{eq_aom2} to estimate
\[A_\Omega + \newkg \cdot \mathcal{H}^{2d-1}(\partial \Omega)\le A_\Omega+ \gamma  h= k.\]
Thus, by Lemma~\ref{lemplunge}, $\lambda_k\le 1/2$. As before, we may assume that $\lambda_k>0$, since otherwise \eqref{eqeig3} holds trivially.
Applying Theorem~\ref{th_2} for $0<\delta\nearrow \lambda_k\le 1/2$ and denoting $\theta=2d/(2d+s-1)$, we obtain
\begin{align*}
k-A_\Omega &\le \lim_{\delta\nearrow \lambda_k} \big| \#\sett{\lambda \in \sigma(\lo_{ \Omega}): \lambda > \delta} - \abs{\Omega} \big|
\le \gamma  \lim_{\delta\nearrow \lambda_k} \delta^{-\theta}	(\log(C_g/\delta))^{1-\theta}
\\ & = \gamma \lambda_k^{-\theta}	(\log(C_g/\lambda_k))^{1-\theta}  \le \frac{\gamma}{\varepsilon^{1-\theta}C_g^\theta}  (C_g/\lambda_k)^{\theta+\varepsilon(1-\theta)},
\end{align*}
for every $\varepsilon>0$, where we used that $\log(x)\le x^\varepsilon/\varepsilon$ for $x>0$. Rearranging the last expression we arrive at
\begin{align}\label{eqh2}
\left(\varepsilon^{1-\theta} C_g^\theta h\right)^{1/(\theta+\varepsilon(1-\theta))}\le \frac{C_g}{\lambda_k}.
\end{align}
(As in the proof of Lemma \ref{lemplunge}, we assume without loss of generality that $\mathcal{H}^{2d-1}(\partial \Omega)>0$, since, otherwise, the conclusions hold trivially.) Choosing $\varepsilon=1/(1+\log (C_g h))\le 1$,
\begin{align*}
C_g \Big(\frac{ h}{(1+\log (C_g h))^{1-\theta}}\Big)^{1/\theta}&=\left(\varepsilon^{1-\theta}C_g^\theta h\right)^{1/(\theta+\varepsilon(1-\theta))} 
\left(\varepsilon^{1-\theta} C_g^\theta h\right)^{\varepsilon(1-\theta)/\theta(\theta+\varepsilon(1-\theta))}
\\&\le \left(\varepsilon^{1-\theta} C_g^\theta h\right)^{1/(\theta+\varepsilon(1-\theta))} 
(C_g h)^{(1-\theta)/(1+\log (C_g h))\theta^2}
\\ &\le \left(\varepsilon^{1-\theta} C_g^\theta h\right)^{1/(\theta+\varepsilon(1-\theta))}  e^{(1-\theta)/\theta^2}
\\ &\le \left(\varepsilon^{1-\theta} C_g^\theta h\right)^{1/(\theta+\varepsilon(1-\theta))}  e^{ 1/\theta^2}.
\end{align*}
Combining the last estimate with \eqref{eqh2} yields \eqref{eqeig3}. 

Towards \ref{coro_2ii}, let $k=A_\Omega-\gamma h\in\N$ with $h\ge 1$ and use \eqref{eq_aom2} to estimate
\[k= A_\Omega - \gamma h  \le A_\Omega -\newkg \cdot\mathcal{H}^{2d-1}(\partial \Omega).\]
From Lemma~\ref{lemplunge} it follows that $\lambda_k\ge 1/2$. 
Applying Theorem~\ref{th_2}, we get
\begin{align*}
A_\Omega-k &\le |\Omega|-(k-1)\le  \big|  \abs{\Omega} - \#\sett{\lambda \in \sigma(\lo_{ \Omega}): \lambda > \lambda_k}  \big|
\\ &\le \gamma    (1-\lambda_k)^{-\theta}	(\log(C_g/(1-\lambda_k)))^{1-\theta}
\\ &\le \frac{\gamma}{\varepsilon^{1-\theta}C_g^\theta} (C_g/(1-\lambda_k))^{\theta+\varepsilon(1-\theta)}, \qquad \varepsilon>0.
\end{align*}
Proceeding exactly as before we arrive at
\begin{align*}
\left(\frac{h}{(1+\log (C_g h))^{1-\theta}}\right)^{1/\theta} \le  \frac{ e^{ 1/\theta^2}}{1-\lambda_k},
\end{align*}
which proves \eqref{eqeig4}.
\end{proof}

\subsection{Almost sharpness of the bounds}
\label{sec_disc}

Let us show that the exponents 1 and $2d-1$ of $(\log \tau)^\beta$ in \eqref{eq_t1} can not be improved. To this end we first inspect the special case of $d=1$, $g$ the Gaussian window $g(t)=2^{1/4} e^{-\pi t^2}$, and $\Omega$ the disk 
$\Omega=B_R(0)\subseteq \R^2$. As shown in \cite{da88}, the eigenvalues of $\lo_{ \Omega}$ are
\[\lambda_{k}=\frac{\gamma(k+1,\pi R^2)}{k!},\]
where 
\[\gamma(s,x)=\int_0^x t^{s-1}e^{-t} \, dt,\]
is the lower incomplete gamma function.

Introducing the random variable $X\sim\Gamma(k+1,1)$ we then have
\[\lambda_{k}=\frac{\gamma(k+1,\pi R^2)}{k!}=P(X\le \pi R^2).\]
From \cite[Theorem~5]{ZZ} we see that for $M> 1$ there are constants $a=a(M),b=b(M)>0$ such that for $\pi R^2\le k+1\le M \pi R^2$,
\[ae^{-b\frac{(k+1-\pi R^2)^2}{k+1}}\le \lambda_{k}=P(X\le \pi R^2)\le e^{-\frac{(k+1-\pi R^2)^2}{2(k+1)}}.\]
Hence, Corollary \ref{coro_1} almost recovers the exponential decay in this region.

Furthermore, a straightforward calculation shows that there exist constants $c,C>0$ such that
\begin{align}\label{eq_a}
\#\sett{k \in \N: \lambda_k > \delta}\ge \pi R^2 + c \sqrt{\log(1/\delta)}R, \qquad
C\le \sqrt{\log(1/\delta)}\le R.
\end{align}
On the other hand, for $ R \ge 1$,
\[\lambda_k \ge \int_0^1 \frac{x^k e^{-x}}{k!} \, dx\gtrsim \frac{1}{k+1!}\gtrsim k^{-k}.\]
A simple computation then shows that
\begin{align*}
\#\sett{k \in \N: \lambda_k > \delta}\gtrsim  \frac{\log(1/\delta)}{\log\log(1/\delta)}.
\end{align*}
In particular,
\begin{align}\label{eq_b}
\#\sett{k \in \N: \lambda_k > \delta} -\pi R^2\gtrsim  \frac{\log(1/\delta)}{\log\log(1/\delta)}, \qquad 1\le R=o\left(\Big(\frac{\log(1/\delta)}{\log\log(1/\delta)}\Big)^{1/2}\right).
\end{align}

\emph{In both regimes, the estimate \eqref{eq_a} and \eqref{eq_b} saturate the upper bound \eqref{eq_good} of Corollary~\ref{coro_dil} up to  $\log \log$ factors.}

The analysis extends to higher dimensions. The eigenvalues of $L_\Omega$ for $\Omega=B_R(0)^d\subseteq \R^{2d}$ are
$\mu_k=\prod_{j=1}^{d}\lambda_{k_j}, k\in\N^d$, where $\lambda_{k}$ are the eigenvalues of the one-dimensional case \cite{da88}. Hence, for $C\le \sqrt{\log(1/\delta)/d}\le R$,
\begin{align*}
	\#\sett{k \in \N^d: \mu_k > \delta}
	&\ge \#\sett{k \in \N: \lambda_k > \delta^{1/d}}^d\ge \left(\pi R^2 + c \sqrt{\log(\delta^{-1/d})}R\right)^d
\\ & \ge (\pi R^2)^d+c_d R^{2d-1}\sqrt{\log(\delta^{-1})},
\end{align*}
while, for $1\le  R =o(\sqrt{\log(1/\delta)/\log\log(1/\delta)})$,
\begin{align*}
	\#\sett{k \in \N^d: \mu_k > \delta}-(\pi R^2)^d
	&\ge \#\sett{k \in \N: \lambda_k > \delta^{1/d}}^d-(\pi R^2)^d\gtrsim \Big(\frac{\log(1/\delta)}{\log\log(1/\delta)}\Big)^d.
\end{align*}
As before, the bound in \eqref{eq_good} is attained up to $\log \log$ factors in these two regions. In particular, the exponents 1 and $2d-1$ of $(\log \tau)^\beta$ in \eqref{eq_t1} can not be improved.
The analysis for $\delta$ close to 1 is similar.

\subsection{Modulation spaces, proof of Remark \ref{rem_m1s}}\label{sec_m1s}
The function $g$ belongs to the \emph{modulation space} $M^1_{s}(\mathbb{R}^d)$, $s>0$, if for some (or any) non-zero Schwartz function $\varphi$,
\begin{align*}
\lVert g \rVert_{M^1_{s}} :=
\int_{\mathbb{R}^{2d}} (1+|z|)^{s} |V_\varphi g(z)| \, dz <\infty,
\end{align*}
see, e.g., \cite{benyimodulation}. 
The following folkloric argument allows us to show that the hypothesis of Theorem~\ref{th_2} concerning $g$ is satisfied, whenever $g \in M^1_{s}(\mathbb{R}^d)$ with $s \geq 1$. The change of window formula for the short-time Fourier transform \cite[Lemma 11.3.3]{grbook} implies
\begin{align*}
|V_gg| \leq |V_\varphi g| * |V_g \varphi|.
\end{align*}
Since $|V_gg(z)| \leq \lVert g \rVert_2^2=1$ and $|V_g \varphi(z)|=|V_\varphi g(-z)|$, we conclude that
\begin{align*}
&\int_{\mathbb{R}^{2d}} (1+|z|)^{s} |V_g g(z)|^2 \,dz
\leq \int_{\mathbb{R}^{2d}} (1+|z|)^{s} |V_g g(z)| \,dz
\\
&\qquad\leq 
\int_{\mathbb{R}^{2d}} \int_{\mathbb{R}^{2d}}
(1+|w|)^{s} |V_\varphi g(w)| (1+|z-w|)^{s} |V_\varphi g(w-z)|
\,dw \,dz = \lVert g \rVert_{M^1_{s}}^2 <\infty.
\end{align*}

\bibliographystyle{abbrv}
\bibliography{eigentf}
\end{document}